\documentclass[12pt]{amsart}
\usepackage{amssymb, amsmath, amsthm}
\usepackage{color}
\usepackage{url}
\usepackage{booktabs}
\usepackage{hyperref}
\usepackage{url}

\newcommand{\thmref}[1]{Theorem~\ref{#1}}

\newtheorem{theorem}{Theorem}[section]
\newtheorem{thm}[theorem]{Theorem}
\newtheorem{lemma}[theorem]{Lemma}

\newtheorem{conj}[theorem]{Conjecture}

\theoremstyle{definition}

\theoremstyle{remark}

\numberwithin{equation}{section}

\begin{document}

\title[Conjectures on divisors of $x^n-1$]{Some conjectures on the maximal
height of divisors of $x^n-1$}

\author{Nathan C. Ryan}
\address{Mathematics Department, Bucknell University, Lewisburg, PA 17837}
\email{nathan.ryan@bucknell.edu}

\author{Bryan C. Ward}
\email{bryan.ward@bucknell.edu}

\author{Ryan Ward}
\email{ryan.ward@bucknell.edu}

%    General info

\date{February 24, 2009 and, in revised form, \today.}

%\dedicatory{This paper is dedicated to our authors.}


\begin{abstract}
Define $B(n)$ to be the largest height of a polynomial in $\mathbb{Z}[x]$
dividing $x^n-1$.  We formulate a number of conjectures related to the
value of $B(n)$ when $n$ is of a prescribed form.  Additionally, we prove a
lower bound for $B(p^aq^b)$ where $p,q$ are distinct primes.
\end{abstract}

\maketitle

\section{Introduction}\label{sec:intro}

The height $H(f)$ of a polynomial $f$ is the largest coefficient of $f$ in
absolute value.  Let 
\[
\Phi_n(x) = \prod_{\substack{1\leq a\leq n\\(a,n)=1}}(x-e^{2\pi ia/n})
\]
be the $n$th cyclotomic polynomial.  For example, for a prime $p$, we have
$\Phi_p(x)=1+x+\cdots+x^{p-1}$. 

Define the function $A(n):= H(\Phi_n(x))$.  This function was originally
studied by Erd\"{o}s and has been much studied since then.  The following
fact reduces the study of $A(n)$ to square-free $n$:
\begin{equation}\label{eqn:basic_fact}
\Phi_{np}(x) = \frac{\Phi_n(x^p)}{\Phi_n(x)}\text{ if $p\nmid n$}\text{ and }\Phi_{np}(x) = \Phi_n(x^p) \text{ if $p\mid n$}.
\end{equation}

The variant we study in the present paper was first defined in
\cite{PomeranceRyan} and studied further in \cite{Kaplan}.  In
\cite{PomeranceRyan} the function
\[
B(n) = \max\{H(f): f|x^n-1\text{ and } f\in\mathbb{Z}[x]\}
\]
is defined and a fairly good asymptotic bound is found.  Also in
\cite{PomeranceRyan} there are two explicit formulas for $n$ of a certain
form: it is shown that $B(p^k)=1$ and $B(pq)=\min\{p,q\}$.  In the present
paper, for $n$ of a prescribed form, we are interested in finding explicit
formulas for $B(n)$, discovering bounds for $B(n)$, determining which
divisors of $x^n-1$ have height $B(n)$ and understanding the image of
$B(n)$.  One might consider the present paper a continuation of
\cite{Kaplan}.  In \cite{Kaplan}, it is shown that $B(p^2q)=\min\{p^2,q\}$.
Additionally, the author found upper bounds for $B(n)$.  Moreover, he found
a better upper bound as well as a lower bound for $B(pqr)$, where $p<q<r$
are three distinct primes.    

Our main theoretical result is a lower bound for $B(p^aq^b)$ but most of the content of the paper is various conjectures about $B(n)$ of the kind described above.  The conjectures are verified by extensive data computed in Sage \cite{Sage} and tabulated in \cite{cyclo-site}.

%Implicit in the above is that there is no obvious relationship between $B(n)$ and $B(n')$ when $n'$ is the square-free part of $n$.  In particular, there is no known formula similar to \eqref{eqn:basic_fact}.

The paper is organized in follows.  In the next section we describe our computations: the method and the scale.  The first of the subsequent two sections is about $B(n)$ for $n$ that are divisible by two distinct primes.  We give a reasonably good lower bound for such $B(n)$ and a few conjectures about $B(pq^b)$.  The third section investigates what happens when 3 or more primes divide $n$.  We conclude the paper with two further variants on the arithmetic function $B(n)$.  For the first of these two variants, related data have also been tabulated in \cite{cyclo-site}.

%\subsection{Database design}

%MySQL is a relational database which stores data in tables composed of a
%fixed number of columns and a variable number of rows.  Each column is
%given a distinct type so as to make queries or lookups on the database more
%efficient.

%

%\bryan{Should I include anything here about the length or the ratio?}
%\nathan{I don't think so}
%Because of the restriction on the number of columns a table can have, we
%have created several tables and defined relationships among different
%rows in different tables so we can map a variable number of data points to
%a single row in another table.  This is accomplished through the common
%construct called a \texttt{foreign key}.  \nathan{Describe a little more about this construct.  Why, e.g., is it called that?  What in the rest of the paragraph is the foreign key?}  This allows us to keep one table
%indexed by $n$ which stores the maximal height, and define separate tables
%which contain rows for each prime factor or set of cyclotomic polynomials
%yielding the maximal height.  Each row in the table of sets of cyclotomic
%polynomials is then linked to by another table containing a variable number
%of rows which define the set of all cyclotomic polynomials that yield this
%maximal height.

%\bryan{This obviously needs a diagram}

%Our program uses Sage the open source mathematics software~\cite{sage}, to
%compute the $B$ function.  We then interface with the MySQL server using
%Storm, an Object Relational Mapper (ORM) for Python built upon the MySQLdb
%API.

%\subsection{Contents of Database}

\section{Computations}

Much of what is included in the present paper is the result of a great deal of
machine computation.  The function $B(n)$ is very difficult to
compute.  The best way we know to compute $B(n)$ is to do the following: observe that any $f$ that
would give a maximal height is a product of cyclotomic polynomials since
\[
x^n-1 = \prod_{d|n} \Phi_d(x).
\]
So, to compute $B(n)$ we need to compute the set of divisors of $n$ and its
power set.  We then iterate over the power set, multiplying the corresponding
cyclotomic polynomials in each set.  The largest height among the
polynomials in this very long list is the value of $B(n)$.  We distributed
the computation over more than 30 processors and it took several months.  The code was
implemented in Sage \cite{Sage}.  As a result of our computations we were
able to formulate the conjectures below. 

A few words about the database that contains the results of our computations are in order.  The data in this database can be accessed via \cite{cyclo-site}.  We store all of the data we see to be useful to formulate conjectures about $B(n)$.  This includes $n$, $B(n)$, and the set of sets of cyclotomic polynomials which multiply to yield
the maximal height.

We have computed $B(n)$ for almost 300000 distinct $n$.  These
computations have taken~30 processors several months to compute (for example $B(720)$ took 113 hours to compute and $B(840)$ took 550 hours to compute) on various systems at Bucknell University: for example, many were run on  a cluster node with dual quad core 3.33ghz xeons with 64GB of ram.  We have
computed $B(n)$ for $n$ with $4$ or fewer prime factors, and for $n$ as
large as $56796482$.  In particular, we have computed $B(n)$ for every $n$ less than 1000.  We present in the next section the conjectures we have formulated based on
these computational data and freely offer access to these data
at \cite{cyclo-site}.  We note that far less comprehensive computations have been done in \cite{Abbott} and a smaller set of data can be found at \cite{oeis}.  We summarize the data in the database that we use to verify our conjectures in Tables \ref{tbl:datapq} and \ref{tbl:datapqr}.

\begin{table}
\begin{tabular}{|l l ll|}\hline
$n$ & Ranges & Data Points & Relevant Conjectures\\\hline
$p^2q^2$ & $2\leq p<q<60$ & 463 & Conjecture \ref{conj:p2q2}\\\hline
$2q^b$ & $2<q<300$, $b = 2$ & 96 & Conjecture \ref{conj:pq}\\%\hline
$2q^b$ & $2<q<100$, $b = 3$ & 24 & \\%\hline
$2q^b$ & $2<q<75$, $b = 4$ & 20 &\\%\hline
$2q^b$ & $2<q<10$, $b = 5$ & 4 &\\%\hline
$2q^b$ & $2<q<10$, $b = 6$ & 4 &\\%\hline
$2q^b$ & $q \in \{3,5\}$, $b = 7$ & 2 &\\%\hline
  $2q^b$ & $q \in \{3,5\}$, $b = 8$ & 2 &\\\hline
$pq^b$ & $2<p<q<85$, $b=3$ & 301 &Conjecture \ref{conj:pq}\\
$pq^b$& $2<p<q<35$, $b=4$ & 92 &Conjecture \ref{conj:pq-range}\\%\hline
$pq^b$& $2<p<q<15$, $b=5$ & 14 &\\%\hline
$pq^b$& $2<p<q<10$, $b=6$ & 13 &\\\hline
\end{tabular}
\caption{Data available at \cite{cyclo-site} and used in verifying the
conjectures in Section \ref{sec:pq}}
\label{tbl:datapq}
\end{table}

\begin{table}
\begin{tabular}{|l l ll|}\hline
$n$ & Ranges & Data Points &Relevant Conjecture\\\hline
$pqr$ & $2\leq p<q<r<150$ & 55530 & Conjecture \ref{conj:pqr} \\%\hline
$pqrs$ & $2\leq p<q<r<s<15$ & 1045 &  \\\hline
$pqr^b$ & $2\leq q <r<50$, $b=2$ & 1490 & Conjecture \ref{conj:pqr-range}\\%\hline
$pqr^b$ & $2\leq q<r<35$, $b=3$ & 171 & \\%\hline
$pqr^b$ & $2\leq q<r<35$, $b=4$ & 13 & \\\hline
\end{tabular}
\caption{Data available at \cite{cyclo-site} and used in verifying the
conjectures in Section \ref{sec:pqr}.}
\label{tbl:datapqr}
\end{table}

\section{When $n$ is divisible by two primes}\label{sec:pq}

Evaluation of the function $A(p^aq^b)$ is rather straightforward.  To see
that $A(p^aq^b)=1$, one can write down an explicit formula for
$\Phi_{pq}(x)$ (see, e.g., \cite{Lam}) and then use \eqref{eqn:basic_fact}.
The situation for $B(p^aq^b)$ is not all like the situation for
$A(p^aq^b)$.

The best general result we have is:
\begin{thm}\label{thm:lowerbound}
$B(p^aq^b)\geq \min\{p^a,q^b\}$
\end{thm}

Before proving the theorem, we prove the following:

\begin{lemma}
For any integer $n$ and prime $p$:
\begin{equation*}\label{assertion}
	\prod^{n}_{k=1}\Phi_{p^{k}}(x) = \sum^{p^{n}-1}_{c = 0}x^{c}
\end{equation*}
\end{lemma}

\begin{proof}
The lemma follows from induction on $n$ and an application of \eqref{eqn:basic_fact} to $\Phi_{p^{n+1}}(x)$.
\end{proof}

%Theorem: $B(p^{a}q^{b}) \geq min\{p^{a},q^{b}\}$.\\
\begin{proof}[Proof of \thmref{thm:lowerbound}]  Consider the following polynomial, a divisor of $x^{p^{a}q^{b}} - 1$:

\begin{equation}\label{product}
	\left(\prod^{a}_{i=1}\Phi_{p^{i}}(x)\right)\left(\prod^{b}_{j=1}\Phi_{q^{j}}(x)\right) = \left(\sum^{p^{a}-1}_{c = 0}x^{c}\right)\left(\sum^{q^{b}-1}_{k = 0}x^{k}\right),
\end{equation}
where the equality follows from the lemma.

 %We consider how the coefficient of
%\eqref{product} of $x^{p^a-1}$. %The polynomial
%\begin{equation*}\label{poly1}\sum^{p^{a}-1}_{c = 0}x^{c}\end{equation*}
%has coefficient $1$ for all monomials $x^i$ where $i\in\{0,\ldots,p^{a-1}\}$. Also, the
%polynomial
%\begin{equation*}\label{poly2}\sum^{q^{b}-1}_{k = 0}x^{k}\end{equation*}
%has coefficient $1$ for all monomials $x^k$ for $k\in \{0,\ldots,p^{a-1}\}$. %here we use the fact that $p^a\leq $q^b$  

Assume that $p^{a} < q^{b}$.  The coefficient
of $x^{p^a-1}$ in \eqref{product} can be seen to be
\[
\sum_{k = 0}^{p^a-1} (1)(1)
\]
since each polynomial in the product has all its coefficients equal to 1.
%the same as the number of of
%possible combinations of exponents from \eqref{poly1} and \eqref{poly2}
%that sum to $p^{a} - 1$ because all the coefficients in \eqref{poly1} and
%\eqref{poly2} are $1$. Since $p^{a}$ combinations are achieved by adding
%elements in $\{0,\ldots,p^{a-1}\}$ and $\{0,\ldots,p^{a-1}\}$ in pairs, the
%coefficient of \eqref{product} with degree $p^{a} - 1$ is $p^{a}$.
Thus:
\begin{equation*}
	B(p^{a}q^{b}) \geq H\left(\left(\prod^{a}_{i=1}\Phi_{p^{i}}(x)\right)\left(\prod^{b}_{j=1}\Phi_{q^{j}}(x)\right)\right) \geq p^{a}.
\end{equation*}
A similiar argument can show, that if $q^{b} < p^{a}$, then:
\begin{equation*}
	B(p^{a}q^{b}) \geq H\left(\left(\prod^{a}_{i=1}\Phi_{p^{i}}(x)\right)\left(\prod^{b}_{j=1}\Phi_{q^{j}}(x)\right)\right) \geq q^{b}.
\end{equation*}

Hence:
\begin{equation*}B(p^{a}q^{b}) \geq \min\{p^{a},q^{b}\}\end{equation*}
\end{proof}

We observe that this bound is surprisingly good for the data we have
computed.  Of the 5396 $n$ in the database of the form $p^aq^b$ (for
$(a,b)\not\in \{(1,1),(1,2),(2,1)\}$), $B(n) =\min\{p^a,q^b\}$ a majority
of the time.

By means of a thorough case-by-case analysis, one can find an explicit
formula for $B(pq^2)$ \cite[Theorem 6]{Kaplan} where $p$ and $q$ are
distinct primes.  The proof proceeds by computing the height of every
possible divisor of $x^{pq^2}-1$ and identifying which of those is largest.
In that spirit we make the note of the following:
\begin{conj}\label{conj:p2q2}
Let $p<q$ be primes.  Then $B(p^2q^2)$ is the larger of
$H(\Phi_p(x)\Phi_q(x)\Phi_{p^2q}(x)\Phi_{pq^2}(x))$ and
$H(\Phi_p(x)\Phi_q(x)\Phi_{p^2}(x)\Phi_{q^2}(x))$.
\end{conj}

For example, $B(3^2\cdot 5^2) = H(\Phi_{3}\Phi_{5}\Phi_{3^2\cdot 5}\Phi_{3\cdot
5^2}) \neq H(\Phi_{3}\Phi_{5}\Phi_{3^2}\Phi_{5^2})$ and
$B(5^2\cdot 11^2) = H(\Phi_{5}\Phi_{11}\Phi_{5^2}\Phi_{11^2}) \neq
H(\Phi_{5}\Phi_{11}\Phi_{5^2\cdot 11}\Phi_{5\cdot 11^2})$.

In addition to not having a proof for this conjecture, we also lack an
explicit formula for the height of the polynomial.  The conjecture has been
checked for the primes indicated in Table \ref{tbl:datapq}.

An even more difficult problem is to deduce a formula for $n$ of a more
arbitrary form.  For example, our computations suggest the following
conjecture.
\begin{conj}\label{conj:pq}
Let $p<q$ be odd primes.
\begin{enumerate}
\item For any positive integer $b$, $B(2q^b)=2$.
\item Suppose $b>2$.  Then $B(pq^b)>p$.
\end{enumerate}
\end{conj}
The difficulty here is that it is not feasible to do a case by case
analysis as described above. 

We have computed data verifying the first part of the conjecture as
indicated in Table \ref{tbl:datapq}.  The cases $b=1$ and $b=2$ in the
first part are theorems in \cite{PomeranceRyan} and \cite{Kaplan},
respectively.  We have verified the second half of the conjecture as
indicated in Table \ref{tbl:datapq}.

The previous conjectures deals with what values of $B(pq^b)$ you get when
you have two fixed primes and let one of the exponents vary.  A related
question is what happens when you have one fixed prime and two fixed
exponents.

\begin{thm}\label{thm:range}
Fix a prime $p$ and positive integers $a$ and $b$.  Then $B(p^aq^b)$ takes
on only finitely many values as $q$ ranges through the set of primes.
\end{thm}
\begin{proof}
This is a rephrasing of a special case of \cite[Theorem 4]{Kaplan}.
\end{proof}

As a result of investigating this theorem computationally, we make the
following observation:

\begin{conj}\label{conj:pq-range}

For a fixed odd prime $p$ and fixed positive integer $b$, the finite list
of values $B(p q^b)$ as $p<q$ varies are all divisible by $p$.
\end{conj}

We have checked this for the same range as which we have checked the second
half of Corollary \ref{conj:pq}.  We observe that $B(7^283^2)=64$, showing
that the hypothesis on the factorization of $n$ as $pq^b$ is necessary.

\section{When $n$ is divisible by more than two primes}\label{sec:pqr}

As noted in \cite[p. 2687]{Kaplan}, one of the products
$\Phi_p(x)\Phi_q(x)\Phi_r(x)\Phi_{pqr}(x)$ or
$\Phi_1(x)\Phi_{pq}(x)\Phi_{pr}(x)\Phi_{qr}(x)$ appears to give the largest
height.  The majority of the time the first product gives the largest
height.  According to out data, of the 27492 $n$ of the form $pqr$ we have
computed, the vast majority of the time the first product does give the
maximal height while the second product only gives the maximal height
only around half of the time (often they both give the maximal height).  In
general, one can make the following conjecture
\begin{conj}\label{conj:pqr}
Let $n=p_1\cdots p_t$ be square free. Then $B(n)$ is given by either
\[
\prod_{d\mid n, \omega(d) \equiv 1 \pmod{2}}\Phi_d(x) \text{ or }
\prod_{d\mid n, \omega(d) \equiv 0 \pmod{2}}\Phi_d(x)
\]
where $\omega(d)$ is the number of primes dividing $d$.
\end{conj}
The conjecture is true when $t=1$ and $t=2$ \cite[Lemma
2.1]{PomeranceRyan}.  Our data supports the conjecture for $n$ as listed in
Table \ref{tbl:datapqr}.

For $n$ that are odd, the analogue to Conjecture \ref{conj:pq-range} would
be: $B(pqr^b)$ is divisible by $p$.  This statement is false for squarefree $n$: $B(3\cdot
31\cdot 1009)=599$ which is not divisible by 3.  On the other hand, we can make the following
conjecture: 

\begin{conj}\label{conj:pqr-range}
Let $n=pqr^b$ where $p<q<r$, and $b>1$. Then $B(n)$ is
divisible by $p$.  Moreover, $B(n)>p$.
\end{conj}
We have evidence for this conjecture as indicated in Table
\ref{tbl:datapqr}.  This conjecture is very much analogous to
Conjecture \ref{conj:pq} and Conjecture \ref{conj:pq-range}.

%As also observed in \cite{Kaplan}, for square-free $n$ of the form $n=pq$ and $n=pqr$, the
%following holds:  a good lower bound for $B(n)$ is attained by the product 
%\[
%\prod_{\omega(d) \equiv 1 \pmod{2}}\Phi_d(x)
%\]
%where $\omega(d)$ is the number of primes dividing $d$.  In particular,
%\begin{align*}
%\Phi_{p}(x)\Phi_{q}(x) &\text{ has height } B(pq)\\
%\Phi_p(x)\Phi_q(x)\Phi_{pqr}(x) &\text{ has height } B(pqr)
%\end{align*}

%We have computational evidence to support the following conjecture:
%\begin{conj}
%Let $n=p_1\cdots p_r$ be square-free.  Then 
%\[
%\prod_{\omega(d) \equiv 1\pmod{2}}\Phi_d(x)
%\]
%has height $B(n)$.
%\end{conj}
%We have checked this for square-free $n$ less than YY with up to XX divisors. 

\section{Conclusions and Future work}

Above we have explicitly described several conjectures about the function $B(n)$.  Implicitly, we have also suggested that proving explicit formulas for $B(n)$, especially by case-by-case analysis, is extremely difficult.  In fact, even conjecturing formulas is difficult.  A new method for proving formulas will be required before more progress can be made.

In addition to the obvious task of proving any of the conjectures included here and developing a new approach to proving these formulas, we propose the following related problems: 
\begin{enumerate}
\item Define the length of a polynomial $f = \sum_{n=0}^d a_nx^n$ to be $L(f)=\sum_{n=0}^d |a_n|$ and let 
\[
C(n):=\max\{L(f):f\mid x^n-1,f\in\mathbf{Z}[x]\}; 
\]
\item let $\mathbb{Q}(\zeta_n)$ be the $n$th cyclotomic field and define the function
\[
D(n):=\max\{H(f):f\in\mathbb{Q}(\zeta_n)[x],\,f\mid x^n-1\text{ and } f \text{ monic} \}.
\]
\end{enumerate}
Can any explicit formulas or bounds be found for these functions?  The database at \cite{cyclo-site} has data related to the first of these two problems.

\bibliographystyle{amsplain}
\bibliography{cyclo}

\end{document}